\documentclass[12pt]{amsart}
\usepackage{amsmath,enumitem}
\usepackage[english,  activeacute]{babel}
\usepackage[latin1]{inputenc}
\usepackage{amssymb}
\usepackage{amsthm}
\usepackage{graphics}
\usepackage{array}
\usepackage{pdflscape}
\usepackage{a4wide}
\usepackage{color, url}
\usepackage{float}

\theoremstyle{plain}
\newtheorem{theorem}{Theorem}

\theoremstyle{definition}

\newcommand{\Z}{{\mathbb Z}}

\def\li{\text{\rm Li}}
\def\lif{\text{\rm Lif}}
\title{A $(p,q)$-Analogue of  Poly-Euler Polynomials and some related polynomials}

\author{Takao Komatsu}
\address{\noindent School of Mathematics and Statistics, Wuhan University, Wuhan 430072,  CHINA}
\email{komatsu@whu.edu.cn}

\author{Jos\'e L. Ram\'{\i}rez}
\address{\noindent Departamento de Matem\'aticas, Universidad Sergio Arboleda, Bogot\'a,  COLOMBIA}
\email{josel.ramirez@ima.usergioarboleda.edu.co}

\author{V\'ictor F. Sirvent}
\address{\noindent Departamento de Matem\'aticas, Universidad Sim\'{o}n Bol\'{i}var, Apartado 89000, Caracas 1086-A, VENEZUELA.}
\email{vsirvent@usb.ve}

\date{\today}
\subjclass[2010]{Primary 11B83; Secondary 11B68, 11B73, 05A19, 	05A15.}
\keywords{Poly-Bernoulli numbers, multi-poly-Bernoulli numbers, multiple polylogarithm function, generating function, combinatorial identities.}

\begin{document}
\begin{abstract}
In the present article, we introduce a $(p,q)$-analogue of the poly-Euler polynomials and numbers by using the $(p,q)$-polylogarithm function. These new sequences are generalizations of the poly-Euler numbers and polynomials. We give several combinatorial identities  and properties of these new polynomials. Moreover, we show some relations with the $(p,q)$-poly-Bernoulli polynomials and  $(p,q)$-poly-Cauchy polynomials.
The $(p,q)$-analogues generalize the well-known concept of the $q$-analogue.
\end{abstract}

\maketitle

\section{Introduction}

The Euler numbers are defined  by the  generating function
\begin{align*}
\frac{2}{e^t+e^{-t}}=\sum_{n=0}^{\infty}E_n\frac{t^n}{n!}.
\end{align*}
The sequence  $(E_n)_{n}$ counts the numbers of alternating $n$-permutations.  A $n$-permutation $\sigma$ is alternating if the $n-1$ differences $\sigma(i+1)-\sigma(i)$ for $i=1, 2, \dots, n-1$ have alternating signs. For example, (1324) and $(3241)$ are alternating permutations (cf. \cite{Comtet}).

The  Euler polynomials are given by the  generating function
\begin{align} \label{peuler}
\frac{2e^{xt}}{e^t+1}=\sum_{n=0}^{\infty}E_n(x)\frac{t^n}{n!}.
\end{align}
Note that $E_n=2^nE_n(1/2)$.

Many kinds of generalizations of these numbers and polynomials have been presented in the literature ~(see, e.g., \cite{Sri}). In particular, we are interested in the poly-Euler numbers and polynomials (cf. ~\cite{Ham, JCK, Jol, Ohno}). \\

The poly-Euler polynomials $E_n^{(k)}(x)$ are defined by the following generating function
\begin{align*}
\frac{2\li_k(1-e^{-t})}{1+e^{t}}e^{xt}=\sum_{n=0}^{\infty}E_n^{(k)}(x)\frac{t^n}{n!}, \quad (k \in \Z),
\end{align*}
where
\begin{align}\label{polyf}
\li_k(t)=\sum_{n=1}^\infty\frac{t^n}{n^k}
\end{align}
is the $k$-th polylogarithm function. Note that if $k=1$, then $\li_1(t)=-\log(1-t)$, therefore
 $E_n^{(1)}(x)=E_{n-1}(x)$ for $n\geq 1$.
\smallskip

It is also possible to define  the poly-Bernoulli and poly-Cauchy numbers and polynomials from the $k$-th polylogarithm function.  In particular, the poly-Bernoulli numbers  $B_n^{(k)}$ were introduced by Kaneko  \cite{Kaneko} by using the following generating function

\begin{align}
\frac{\li_k(1-e^{-t})}{1-e^{-t}}=\sum_{n=0}^{\infty}B_n^{(k)}\frac{t^n}{n!}, \quad (k\in \Z)
\end{align}
If $k=1$ we get $B_n^{(1)}=(-1)^nB_n$ for $n\geq 0$, where $B_n$ are the Bernoulli numbers.  Remember that the Bernoulli numbers $B_n$ are defined by the  generating function
\begin{align*}
\frac{t}{e^t-1}=\sum_{n=0}^{\infty}B_n\frac{t^n}{n!}.
\end{align*}

The poly-Bernoulli numbers and polynomials have been studied in several papers; among other references, see  \cite{Bayad2, Brew,  Kom1, Cenk, Kom4, Kom3}.

The poly-Cauchy numbers of the first kind $c_n^{(k)}$ were introduced  by the first author  in \cite{Komc}. They are defined as follows
\begin{align}
c_{n}^{(k)}=\underbrace{\int_{0}^1 \cdots \int_{0}^1}_{k}(t_1\cdots t_k)_n \,dt_1\cdots dt_k
\end{align}
where $(x)_n=x(x-1)\cdots (x-n+1) (n\geq 1)$ with $(x)_0=1$.
Moreover, its exponential generating function is
\begin{align}
\lif_k(\ln(1+t))=\sum_{n=0}^{\infty}c_n^{(k)}\frac{t^n}{n!}, \quad (k\in \Z)
\end{align}
where
$$\lif_k(t)=\sum_{n=0}^\infty\frac{t^n}{n!(n+1)^k}$$
is the $k$-th polylogarithm factorial function. For more properties about these numbers see for example \cite{Cenk, Kom5, Kom4, Kom3, Kom2, Kom6}.  If $k=1$, we recover the Cauchy numbers $c_n^{(1)}=c_n$.  The Cauchy numbers $c_n$ were introduced in \cite{Comtet}  by the  generating function
\begin{align*}
\frac{t}{\ln(1+t)}=\sum_{n=0}^{\infty}c_n\frac{t^n}{n!}.
\end{align*}

A generalization of the above sequences was done recently in \cite{Kom4}, using the $k$-th $q$-polylogarithm function   and the Jackson's integral.  In particular, the $q$-poly-Bernoulli numbers are defined by
\begin{align}
\frac{\li_{k,q}(1-e^{-t})}{1-e^{-t}}=\sum_{n=0}^{\infty}B_{n,q}^{(k)}\frac{t^n}{n!}, \quad (k\in \Z,  n\geq 0, 0 \leq q <1),
\end{align}
where
$$\li_{k,q}(t)=\sum_{n=1}^\infty\frac{t^n}{[n]_q^k}$$
is the $k$-th $q$-polylogarithm function (cf. \cite{Mansour}), and $[n]_q=\frac{1-q^n}{1-q}$ is the $q$-integer (cf. \cite{Sri}).
Note that  $\lim_{q\to 1}[x]_q=x, \lim_{q\to 1}B_{n,q}^{(k)}=B_n^{(k)}$ and $ \lim_{q\to 1}\li_{k,q}(x)=\li_k(x)$. \\

The $q$-poly-Cauchy numbers of the first kind $c_{n,q}^{(k)}$ are defined by using the Jackson's $q	$-integral (cf. \cite{Andrews})
\begin{align}
c_{n,q}^{(k)}=\underbrace{\int_{0}^1 \cdots \int_{0}^1}_{k}(t_1\cdots t_k)_nd_{q}t_1\cdots d_{q}t_k
\end{align}
where
$$\int_0^xf(t)d_qt=(1-q)x\sum_{n=0}^\infty f(q^nx)q^n.$$
Moreover, its exponential generating function is
$$\lif_{k,q}(\ln(1+t))=\sum_{n=0}^{\infty}c_{n,q}^{(k)}\frac{t^n}{n!}, \quad (k\in \Z)
$$
where
\begin{align} \label{genqcauhcy}
\lif_{k,q}(t)=\sum_{n=0}^\infty\frac{t^n}{n![n+1]_q^k}
\end{align}
is the $k$-th $q$-polylogarithm factorial function (cf. \cite{Kom4, Kim}). Note that  $\lim_{q\to 1}c_{n,q}^{(k)}=c_n^{(k)}$ and $ \lim_{q\to 1}\lif_{k,q}(t)=\lif_k(t)$.

\smallskip

In this paper, we introduce a $(p,q)$-analogue of the poly-Euler polynomials by
\begin{align} \label{polyeuler}
\frac{2\li_{k,p,q}(1-e^{-t})}{1+e^{t}}e^{xt}=\sum_{n=0}^{\infty}E_{n,p,q}^{(k)}(x)\frac{t^n}{n!}, \quad (k \in \Z)
\end{align}
with $p$ and $q$ real numbers such that $0<q<p \leq 1$, and $$\li_{k,p,q}(t)=\sum_{n=1}^\infty\frac{t^n}{[n]_{p,q}^k}$$
is an extension of the $q$-polylogarithm function and we call it the $(p,q)$-polylogarithm function. The polynomials  $E_{n,p,q}^{(k)}(0):=E_{n,p,q}^{(k)}$ are called $(p,q)$-poly-Euler numbers.  The polynomial $[n]_{p,q}=\frac{p^n-q^n}{p-q}$ is the $n$-th $(p,q)$-integer~(cf. \cite{Hou, Jag, Sah}), it was introduced in the context of set partition statistics~(cf.~\cite{WW}).
Note that $\lim_{p\to 1}[n]_{p,q}=[n]_q$ and  $ \lim_{p\to 1}\lif_{k,p,q}(t)=\lif_{k,q}(t)$.

\smallskip

 As we already mentioned the $(p,q)$-analogues are an extension of the $q$-analogues, and coincide in the limit  when $p$ tends to $1$.
 The $(p,q)$-calculus was studied in~\cite{CJ}, in  connection with quantum mechanics.
 Properties of the $(p,q)$-analogues of the binomial coefficients were studied in~\cite{Corcino}.
 The $(p,q)$-analogues of hypergeometric series, special functions, Stirling numbers, Hermite polynomials have been studied before, see for instance~\cite{Jag,Nish, SS,SW}.\\

The paper is divided in two parts. In Section 2  we show several combinatorial identities of the $(p,q)$-poly-Euler polynomials. Some of them involving the classical Euler polynomials and another special numbers and polynomials such as the Stirling numbers of the second kind, Bernoulli polynomials of order $s$, etc.  In Section 3 we introduce the $(p,q)$-poly-Bernoulli polynomials  and $(p,q)$-poly-Cauchy polynomials of both kinds, and we generalize some well-known identities  of the classical Bernoulli and Cauchy numbers and polynomials.

\section{Some Identities of the Poly-Euler polynomials}
In this section, we give several identities of the $(p,q)$-poly-Euler polynomials. In particular,  Theorem \ref{eucls}  shows a relation between the $(p,q)$-poly-Euler polynomials and the classical Euler polynomials. \\

It is possible to give the first values of the  $(p,q)$-polylogarithm function for $k\leq 0$. For example,
\begin{align*}
\li_{0,p,q}(x)&=\frac{x}{1-x},\\
\li_{-1,p,q}(x)&=\frac{x}{(1-p x) (1-q x)},\\
\li_{-2,p,q}(x)&=\frac{x (1+pqx)}{\left(1-p^2 x\right) \left(1-q^2 x\right) (1-p q x)},\\
\li_{-3,p,q}(x)&=\frac{x \left(p^3 q^3 x^2+2 p^2 q x+2 p q^2 x+1\right)}{\left(1-p^3 x\right)
   \left(1-q^3 x\right) \left(1-p^2 q x\right) \left(1-p q^2 x\right)}.
\end{align*}

In general, the $(p,q)$-polylogarithm function for $k\leq 0$ is a rational function. Indeed, let $k$  be a nonnegative integer then
\begin{align*}
\li_{-k,p,q}(x)&=\sum_{n=1}^\infty \frac{x^n}{[n]_{p,q}^{-k}}=\sum_{n=1}^\infty [n]_{p,q}^{k}x^n=\sum_{n=1}^\infty \left(\frac{p^n-q^n}{p-q}\right)^kx^n\\
&=\frac{1}{(p-q)^k} \sum_{n=1}^\infty  \sum_{l=0}^k\binom kl p^{nl}(-q^n)^{k-l}x^n=\frac{1}{(p-q)^k} \sum_{l=0}^k(-1)^{k-l}  \binom kl \frac{p^lq^{k-l}x}{1-p^lq^{k-l}x}.
\end{align*}

Note that from \eqref{polyeuler} we obtain that  $\{E_{n,p,q}^{(k)}(x)\}_{n\geq0}$ is an Appel sequence \cite{Roman}. Therefore,  we have the following basic relations.

\begin{theorem}
If $n\geq 0$  and $k\in \Z$ then
\begin{enumerate}[label={(\roman*)}]
\item $\begin{aligned}[t]
E_{n,p,q}^{(k)}(x) &= \sum_{i=0}^n \binom ni E_{i,p,q}^{(k)}x^{n-i}.
 \end{aligned}$
\item  $\begin{aligned}[t]
E_{n,p,q}^{(k)}(x+y) &= \sum_{i=0}^n \binom ni E_{i,p,q}^{(k)}(x)y^{n-i}.
 \end{aligned}$
\item  $\begin{aligned}[t]
E_{n,p,q}^{(k)}(mx) &= \sum_{i=0}^n \binom ni E_{i,p,q}^{(k)}(x)(m-1)^{n-i}x^{n-i}, \ m\geq 1.
 \end{aligned}$
 \item  $\begin{aligned}[t]
E_{n,p,q}^{(k)}(x+1) - E_{n,p,q}^{(k)}(x)&= \sum_{i=0}^{n-1} \binom ni E_{i,p,q}^{(k)}(x).
 \end{aligned}$

\end{enumerate}

\end{theorem}

\begin{theorem}\label{eucls}
If $n\geq 1$ we have
\begin{align*}
E_{n,p,q}^{(k)}(x)=\sum_{l=0}^\infty \frac{1}{[l+1]_{p,q}^k}\sum_{j=0}^{l+1}\binom{l+1}{j}(-1)^jE_n(x-j).
\end{align*}
\end{theorem}
\begin{proof}
From (\ref{polyf}) and (\ref{polyeuler}) we get
\begin{align*}
\frac{2\li_{k,p,q}(1-e^{-t})}{1+e^{t}}e^{xt}&=\sum_{l=0}^\infty\frac{(1-e^{-t})^{l+1}}{[l+1]^k_{p,q}}\cdot \frac{2e^{xt}}{1+e^t}\\
&=\sum_{l=0}^\infty\frac{1}{[l+1]^k_{p,q}}\sum_{j=0}^{l+1}\binom{l+1}{j}(-1)^j \frac{2e^{(x-j)t}}{1+e^t}\\
&=\sum_{l=0}^\infty\frac{1}{[l+1]^k_{p,q}}\sum_{j=0}^{l+1}\binom{l+1}{j}(-1)^j \sum_{n=0}^\infty E_n(x-j) \frac{t^n}{n!}.
\end{align*}
Comparing the coefficients on both sides, we get the desired result.
\end{proof}

\begin{theorem}
If $n\geq 1$ we have
\begin{align*}
E_{n,p,q}^{(k)}(x)=\sum_{l=0}^\infty\sum_{i=0}^l\sum_{j=0}^{i+1} \frac{2(-1)^{l-i-j}}{[i+1]^k_{p,q}}\binom{i+1}{j}(l-i-j+x)^n.
\end{align*}
\end{theorem}
\begin{proof}
By using the binomial series we get
\begin{align*}
\frac{2\li_{k,p,q}(1-e^{-t})}{1+e^{t}}e^{xt}&=2\left(\sum_{l=0}^\infty(-1)^le^{lt}\right) \left(\sum_{l=0}^\infty \frac{(1-e^{-t})^{l+1}}{[l+1]_{p,q}^k}\right)e^{xt}\\
&=2\sum_{l=0}^\infty\sum_{i=0}^l \frac{(-1)^{l-i}e^{(l-i)t}}{[i+1]_{p,q}^k}(1-e^{-t})^{i+1}e^{xt}\\
&=\left(2\sum_{l=0}^\infty\sum_{i=0}^l \frac{(-1)^{l-i}e^{(l-i)t}}{[i+1]_{p,q}^k}\right) \left(\sum_{j=0}^{i+1}\binom{i+1}{j}(-1)^je^{-tj}e^{xt}\right)\\
&=2\sum_{l=0}^\infty\sum_{i=0}^l\sum_{j=0}^{i+1} \frac{(-1)^{l-i+j}e^{(l-i-j+x)t}}{[i+1]_{p,q}^k}\binom{i+1}{j}\\
&=2\sum_{l=0}^\infty\sum_{i=0}^l \sum_{j=0}^{i+1} \frac{(-1)^{l-i+j}}{[i+1]_{p,q}^k}\binom{i+1}{j}\sum_{n=0}^\infty (l-i-j+x)^n\frac{t^n}{n!}\\
&=\sum_{n=0}^\infty \sum_{l=0}^\infty\sum_{i=0}^l \sum_{j=0}^{i+1} \frac{2(-1)^{l-i+j}}{[i+1]_{p,q}^k}\binom{i+1}{j}(l-i-j+x)^n\frac{t^n}{n!}.
\end{align*}
Comparing the coefficients on both sides, we get the desired result.
\end{proof}

\subsection{Some Relations with Other Special Polynomials}
Jolany et al. \cite{JCK} discovered several combinatorics identities involving generalized poly-Euler polynomials
in terms of Stirling numbers of the second kind $S_2(n,k)$, rising factorial functions $(x)^{(m)}$, falling factorial functions $(x)_m$, Bernoulli polynomials $\mathfrak B_n^{(s)}(x)$ of order $s$, and Frobenius-Euler functions $H_n^{(s)}(x;u)$.   We will give similar expressions in terms of  $(p,q)$-poly-Euler polynomials

Remember that the Stirling numbers of the second kind are defined by
\begin{align} \label{sti}
\frac{(e^x-1)^m}{m!}=\sum_{n=m}^\infty S_2(n,m)\frac{x^n}{n!}.
\end{align}

\begin{theorem}
We have the following  identity
\begin{align}
E_{n,p,q}^{(k)}(x)&=\sum_{l=0}^\infty \sum_{i=l}^n \binom ni S_2(i,l) E_{n-i,p,q}^{(k)}(-l)(x)^{(l)} \label{tid1}
\end{align}
where
\begin{align*}
(x)^{(m)}&=x(x+1)\cdots(x+m-1)\quad(m\ge 1)\quad\hbox{with}\quad (x)^{(0)}=1.
\end{align*}
\end{theorem}
\begin{proof}
From \eqref{polyeuler} and \eqref{sti}, and by the binomial series
$$\frac{1}{(1-x)^c}=\sum_{n=0}^\infty\binom{c+n-1}{n}x^n$$
we get:
\begin{align*}
\frac{2\li_{k,p,q}(1-e^{-t})}{1+e^{t}}e^{xt}&=\frac{2\li_{k,p,q}(1-e^{-t})}{1+e^{t}}(1-(1-e^{-t}))^{-x}\\
&=\frac{2\li_{k,p,q}(1-e^{-t})}{1+e^{t}} \sum_{l=0}^\infty \binom{x+l-1}{l}(1-e^{-t})^l\\
&=\sum_{l=0}^\infty \frac{(x)^{(l)}}{l!}(1-e^{-t})^l \frac{2\li_{k,p,q}(1-e^{-t})}{1+e^{t}} \\
&=\sum_{l=0}^\infty (x)^{(l)}\frac{(e^t-1)^l}{l!} \left(\frac{2\li_{k,p,q}(1-e^{-t})}{1+e^{t}}e^{-tl}\right) \\
&=\sum_{l=0}^\infty (x)^{(l)}\left(\sum_{n=0}^\infty S_2(n,l)\frac{t^n}{n!}\right)  \left(\sum_{n=0}^\infty E_{n,p,q}^{(k)}(-l)\frac{t^n}{n!}\right) \\
&=\sum_{l=0}^\infty (x)^{(l)} \sum_{n=0}^\infty\left(\sum_{i=0}^n \binom ni S_2(i,l) E_{n-i,p,q}^{(k)}(-l)\right)\frac{t^n}{n!} \\
&=\sum_{n=0}^\infty \left( \sum_{l=0}^\infty \sum_{i=l}^n \binom ni S_2(i,l) E_{n-i,p,q}^{(k)}(-l)(x)^{(l)}\right)\frac{t^n}{n!}.
\end{align*}
Comparing the coefficients on both sides, we have \eqref{tid1}. Note that we use the following relation
\begin{align*}\binom{x+l-1}{s}=\frac{(x)^{(l)}}{s!}.\end{align*}\qedhere
\end{proof}

\begin{theorem}
We have the following  identity
\begin{align}
E_{n,p,q}^{(k)}(x)&=\sum_{l=0}^\infty \sum_{i=l}^n \binom ni S_2(i,l) E_{n-i,p,q}^{(k)}(x)_l, \label{tid2}
\end{align}
where
\begin{align*}
(x)_m&=x(x-1)\cdots(x-m+1)\quad(m\ge 1)\quad\hbox{with}\quad (x)_0=1.
\end{align*}
\end{theorem}
\begin{proof}
From \eqref{polyeuler} and \eqref{sti}
\begin{align*}
\frac{2\li_{k,p,q}(1-e^{-t})}{1+e^{t}}e^{xt}&=\frac{2\li_{k,p,q}(1-e^{-t})}{1+e^{t}}((e^{t}-1)+1)^x\\
&=\frac{2\li_{k,p,q}(1-e^{-t})}{1+e^{t}} \sum_{l=0}^\infty \binom{x}{l}(e^{t}-1)^l\\
&=\sum_{l=0}^\infty \frac{(x)_l}{l!}(e^{t}-1)^l \frac{2\li_{k,p,q}(1-e^{-t})}{1+e^{t}} \\
&=\sum_{l=0}^\infty (x)_l \left(\sum_{n=0}^\infty S_2(n,l)\frac{t^n}{n!}\right) \left(\sum_{n=0}^\infty E_{n,p,q}^{(k)}\frac{t^n}{n!}\right) \\
&=\sum_{l=0}^\infty (x)_l \sum_{n=0}^\infty\left(\sum_{i=0}^n \binom ni S_2(i,l) E_{n-i,p,q}^{(k)}\right)\frac{t^n}{n!} \\
&=\sum_{n=0}^\infty \left( \sum_{l=0}^\infty \sum_{i=l}^n \binom ni S_2(i,l) E_{n-i,p,q}^{(k)}(x)_l\right)\frac{t^n}{n!}.
\end{align*}
Comparing the coefficients on both sides, we have \eqref{tid2}. Note that we use the following relation
$$\binom{x}{s}=\frac{(x)_s}{s!}.$$
\end{proof}

 The Bernoulli polynomials $\mathfrak B_n^{(s)}(x)$ of order $s$ are defined by
\begin{align}\label{bergs}
\left(\frac{t}{e^t-1}\right)^s e^{x t}=\sum_{n=0}^\infty\mathfrak B_n^{(s)}(x)\frac{t^n}{n!}.
\end{align}
It is clear that if $s=1$ we recover the classical Bernoulli polynomials.   For some explicit formulae of these polynomials see for example \cite{Liu}.

\begin{theorem}
We have the following  identity
\begin{align}
E_{n,p,q}^{(k)}(x)&=\sum_{l=0}^n\binom nl S_2(l+s,s)   \sum_{i=0}^{n-l} \frac{\binom{n-l}{i}}{\binom{l+s}{s}} \mathfrak B_i^{(s)}(x) E_{n-l-i,p,q}^{(k)}. \label{tid3}
\end{align}
\end{theorem}
\begin{proof}
From \eqref{polyeuler} and \eqref{bergs}
\begin{align*}
\frac{2\li_{k,p,q}(1-e^{-t})}{1+e^{t}}e^{xt}&=\frac{(e^t-1)^s}{s!} \frac{t^se^{xt}}{(e^t-1)^s} \left(\sum_{n=0}^\infty E_{n,p,q}^{(k)}\frac{t^n}{n!}\right) \frac{s!}{t^s} \\
&=\left(\sum_{n=0}^\infty S_2(n+s,s)\frac{t^{n+s}}{(n+s)!}\right) \left(\sum_{n=0}^\infty \mathfrak B_n^{(s)}(x) \frac{t^{n}}{n!}\right)  \left(\sum_{n=0}^\infty E_{n,p,q}^{(k)}\frac{t^n}{n!}\right) \frac{s!}{t^s}\\
&=\left(\sum_{n=0}^\infty S_2(n+s,s)\frac{t^{n+s}}{(n+s)!}\right) \sum_{n=0}^\infty \left(\sum_{i=0}^n \binom{n}{i} \mathfrak B_i^{(s)}(x) E_{n-i,p,q}^{(k)} \right)\frac{t^n}{n!}  \frac{s!}{t^s}\\
&=\sum_{n=0}^\infty \left(\sum_{l=0}^n S_2(l+s,s)\frac{t^{l+s}}{(l+s)!} \sum_{i=0}^{n-l} \binom{n-l}{i} \mathfrak B_i^{(s)}(x) E_{n-l-i,p,q}^{(k)}\frac{t^{n-l}}{(n-l)!} \right)  \frac{s!}{t^s}\\
&=\sum_{n=0}^\infty \left(\sum_{l=0}^n  \binom{n}{l} S_2(l+s,s) \sum_{i=0}^{n-l}  \frac{\binom{n-l}{i}}{\binom{l+s}{s}} \mathfrak B_i^{(s)}(x) E_{n-l-i,p,q}^{(k)} \right)   \frac{t^n}{n!}.
\end{align*}
Comparing the coefficients on both sides, we get \eqref{tid3}.
\end{proof}

The Frobenius-Euler functions $H_n^{(s)}(x;u)$ are defined by
\begin{align}\label{FEuler}
 \left(\frac{1-u}{e^t-u}\right)^s e^{x t}=\sum_{n=0}^\infty H_n^{(s)}(x;u)\frac{t^n}{n!}.
\end{align}

\begin{theorem}
We have the following  identity
\begin{align}
E_{n,p,q}^{(k)}(x)&=\sum_{l=0}^n\frac{\binom nl}{(1-u)^s} \sum_{i=0}^s\binom si (-u)^{s-i}H_l^{(s)}(x;u)E_{n-l,p,q}^{(k)}(i). \label{tid4}
\end{align}
\end{theorem}
\begin{proof}
From \eqref{polyeuler} and \eqref{FEuler}
\begin{align*}
\frac{2\li_{k,p,q}(1-e^{-t})}{1+e^{t}}e^{xt}&=\frac{(1-u)^s}{(e^t-u)^s}e^{xt}\frac{(e^t-u)^s}{(1-u)^s}\frac{2\li_{k,p,q}(1-e^{-t})}{1+e^{t}}\\
&=\frac{1}{(1-u)^s} \left(\sum_{n=0}^\infty H_n^{(s)}(x;u)\frac{t^n}{n!}\right)\sum_{i=0}^s\binom si e^{ti}(-u)^{s-i} \frac{2\li_{k,p,q}(1-e^{-t})}{1+e^{t}}\\
&=\frac{1}{(1-u)^s} \left(\sum_{n=0}^\infty H_n^{(s)}(x;u)\frac{t^n}{n!}\right)\sum_{i=0}^s\binom si (-u)^{s-i} \sum_{n=0}^\infty E_{n,p,q}^{(k)}(i)\frac{t^n}{n!}\\
&=\frac{1}{(1-u)^s} \sum_{i=0}^s\binom si (-u)^{s-i}  \left(\sum_{n=0}^\infty H_n^{(s)}(x;u)\frac{t^n}{n!}\right)\left(\sum_{n=0}^\infty E_{n,p,q}^{(k)}(i)\frac{t^n}{n!}\right)\\
&=\frac{1}{(1-u)^s} \sum_{i=0}^s\binom si (-u)^{s-i} \sum_{n=0}^\infty\left( \sum_{l=0}^n\binom{n}{l} H_l^{(s)}(x;u) E_{n-l,p,q}^{(k)}(i)\right)\frac{t^n}{n!}\\
&= \sum_{n=0}^\infty \left( \frac{1}{(1-u)^s}  \sum_{l=0}^n\binom{n}{l}  \sum_{i=0}^s\binom si (-u)^{s-i} H_l^{(s)}(x;u) E_{n-l,p,q}^{(k)}(i)\right)\frac{t^n}{n!}.\\
\end{align*}
Comparing the coefficients on both sides, we get \eqref{tid4}.
\end{proof}

\section{The $(p,q)$-poly Bernoulli Polynomials and the $(p,q)$-poly poly-Cauchy polynomials}

In this section we introduce the $(p,q)$-poly Bernoulli polynomials by means of the $(p,q)$-polylogarithm function and the $(p,q)$-poly Cauchy  polynomials by using the $(p,q)$-integral.  In general it is not difficult  to extend the results of \cite{Kom4}.

The $(p,q)$-derivative of the function $f$ is defined by  (cf. \cite{Burban, Hou})
$$D_{p,q}f(x)=\begin{cases}
\frac{f(px)-f(qx)}{(p-q)x},& \text{if} \ x\neq 0,\\
f'(0),& \text{if} \ x=0.
\end{cases}
$$
In particular if $p\to 1$ we obtain the $q$-derivative \cite{Andrews}.
The $(p,q)$-integral of the function $f$ is defined by
$$\int_{0}^{x}f(t)d_{p,q}t=
\begin{cases}
(q-p)x\sum_{n=0}^\infty\frac{p^n}{q^{n+1}} f\left(\frac{p^n}{q^{n+1}}x\right), & \text{if} \  |p/q|<1;\\
(p-q)x\sum_{n=0}^\infty\frac{q^n}{p^{n+1}} f\left(\frac{q^n}{p^{n+1}}x\right), & \text{if} \  |p/q|>1.
\end{cases}$$

For example,
$$\int_{0}^{1}t^l d_{p,q}t=\frac{1}{[l+1]_{p,q}}.$$

We introduce the $(p,q)$-poly Bernoulli polynomials by
\begin{align*}
\frac{\li_{k,p,q}(1-e^{-t})}{1-e^{-t}}e^{-xt}=\sum_{n=0}^{\infty}B_{n,p,q}^{(k)}(x)\frac{t^n}{n!}, \quad (k\in \Z).
\end{align*}
In particular, $\lim_{p \to 1} B_{n,p,q}^{(k)}(x)=B_{n,q}^{(k)}(x)$, which are the $q$-poly-Bernoulli polynomials studied recently  in \cite{Kom4}.

The following theorem related the $(p,q)$-poly-Bernoulli  polynomials and  $(p,q)$-poly-Euler  polynomials.
\begin{theorem}
If $n\geq 1$ we have
$$E_{n,p,q}^{(k)}(x) + E_{n,p,q}^{(k)}(x+1)=2B_{n,p,q}^{(k)}(-x) - 2 B_{n,p,q}^{(k)}(1-x).$$
\end{theorem}
\begin{proof}
From the following equality
$$\frac{2\li_{k,p,q}(1-e^{-t})}{1+e^t}(1+e^t)e^{xt}=\frac{2\li_{k,p,q}(1-e^{-t})}{1-e^{-t}}(1-e^{-t})e^{xt}$$
we obtain
\begin{align*}
\sum_{n=0}^\infty E_{n,p,q}^{(k)}(x)\frac{t^n}{n!} + \sum_{n=0}^\infty E_{n,p,q}^{(k)}(x+1)\frac{t^n}{n!} = 2\sum_{n=0}^\infty B_{n,p,q}^{(k)}(-x)\frac{t^n}{n!} - 2\sum_{n=0}^\infty B_{n,p,q}^{(k)}(1-x)\frac{t^n}{n!}.
\end{align*}
Comparing the coefficients on both sides, we get the desired result.
\end{proof}

The weighted Stirling numbers of the second kind, $S_2(n,m,x)$, were defined by Carlitz \cite{Car} as follows
$$\frac{e^{xt}(e^t-1)^m}{m!}=\sum_{n=m}^\infty S_2(n,m,x)\frac{t^n}{n!}.$$

\begin{theorem}\label{polyberrel1}
If $n\geq 1$, we have
\begin{align*}
B_{n,p,q}^{(k)}(x)=\sum_{m=0}^n\frac{(-1)^{m+n}m!}{[m+1]^k_{p,q}}S_2(n,m,x).
\end{align*}
\end{theorem}
\begin{proof}
\begin{align*}
\sum_{n=0}^{\infty}B_{n,p,q}^{(k)}(x)\frac{t^n}{n!}&=\frac{\li_{p,q}(1-e^{-t})}{1-e^{-t}}e^{-xt}\\
&= \sum_{m=0}^{\infty} \frac{(1-e^{-t})^m}{[m+1]_{p,q}^k}e^{-xt}\\
&= \sum_{m=0}^{\infty} \frac{(-1)^mm!}{[m+1]_{p,q}^k}\cdot \frac{(e^{-t}-1)^m}{m!}e^{-xt}\\
&= \sum_{m=0}^{\infty} \frac{(-1)^mm!}{[m+1]_{p,q}^k}\cdot \sum_{n=m}^{\infty}S_2(n,m,x)\frac{(-t)^n}{n!}\\
&= \sum_{n=0}^{\infty}\left(\sum_{m=0}^{\infty}\frac{(-1)^{m+n}m!}{[m+1]_{p,q}^k}S_2(n,m,x)\right)\frac{t^n}{n!}.
\end{align*}
Comparing the coefficients on both sides, we get the desired result.
\end{proof}

The $(p,q)$-poly-Cauchy polynomials of the first kind are defined by
\begin{align}\label{pqcau}
C_{n,p,q}^{(k)}(x)&=\underbrace{\int_{0}^1 \cdots \int_{0}^1}_{k}(t_1\cdots t_k-x)_nd_{p,q}t_1\cdots d_{p,q}t_k.
\end{align}

Note that $\lim_{p\to 1}C_{n,p,q}^{(k)}(x)=C_{n,q}^{(k)}(x)$, i.e., we obtain the $q$-poly-Cauchy polynomials \cite{Kom4, Kim}.\\

Remember that the (unsigned) Stirling numbers of the first kind are defined by
\begin{align} \label{sti1}
\frac{(\ln(1+x))^m}{m!}=\sum_{n=m}^\infty (-1)^{n-m}S_1(n,m)\frac{x^n}{n!}.
\end{align}
Moreover, they satisfy (cf.  \cite{Comtet})
\begin{align}\label{sti1r}
x^{(n)}=x(x+1)\cdots (x+n-1)=\sum_{m=0}^nS_1(n,m)x^m.
\end{align}
The weighted Stirling numbers of the first kind, $S_1(n,m,x)$, are defined  by (\cite{Car})
$$\frac{(1-t)^{-x}(-\ln(1-t))^m}{m!}=\sum_{n=m}^\infty S_1(n,m,x)\frac{t^n}{n!}.$$

\begin{theorem}\label{teocaupq}
If $n\geq 1$, we have
\begin{align}
C_{n,p,q}^{(k)}(x)&=\sum_{m=0}^n(-1)^{n-m}S_1(n,m)\sum_{l=0}^m\binom ml \frac{(-x)^l}{[m-l+1]_{p,q}^k} \label{ec1cauchy}\\
&=\sum_{m=0}^nS_1\left(n,m,x\right)\frac{(-1)^{n-m}}{[m+1]_{p,q}^k}. \label{ec2cauchy}
\end{align}
\end{theorem}
\begin{proof}
By \eqref{pqcau}, \eqref{sti1r}  and  $(x)_n=(-1)^n(-x)^{(n)}$, we have
\begin{align*}
C_{n,p,q}^{(k)}(x)&=\sum_{m=0}^{n}(-1)^{n-m}S_1(n,m)\underbrace{\int_{0}^1 \cdots \int_{0}^1}_{k}(t_1\cdots t_k-x)^m d_{p,q}t_1\cdots d_{p,q}t_k\\
&=\sum_{m=0}^{n}(-1)^{n-m}S_1(n,m)\sum_{l=0}^m\binom ml (-x)^{m-l}\underbrace{\int_{0}^1 \cdots \int_{0}^1}_{k} t_1^l\cdots t_k^l d_{p,q}t_1\cdots d_{p,q}t_k\\
&=\sum_{m=0}^{n}(-1)^{n-m}S_1(n,m)\sum_{l=0}^m\binom ml \frac{(-x)^{m-l}}{[l+1]_{p,q}^k}\\
&=\sum_{m=0}^{n}(-1)^{n-m}S_1(n,m)\sum_{l=0}^m\binom ml \frac{(-x)^{l}}{[m-l+1]_{p,q}^k}.
\end{align*}
Comparing the coefficients on both sides, we get \eqref{ec1cauchy}.  Finally, 
from the following relation (\cite[Eq. (5.2)]{Car})
$$S_1(n,m,x)=\sum_{i=0}^n\binom{m+i}{i}x^iS_1(n,m+i),$$
we have
\begin{align*}
C_{n,p,q}^{(k)}(x)&=\sum_{m=0}^n(-1)^{n-m}S_1(n,m)\sum_{l=0}^m\binom ml \frac{(-x)^l}{[m-l+1]_{p,q}^k} \\
&=\sum_{l=0}^n\sum_{m=l}^n(-1)^{n-m}S_1(n,m)\binom ml \frac{(-x)^l}{[m-l+1]_{p,q}^k}\\
&=\sum_{l=0}^n\sum_{m=l}^{n+l}(-1)^{n-m}S_1(n,m)\binom ml \frac{(-x)^l}{[m-l+1]_{p,q}^k}\\
&=\sum_{l=0}^n\sum_{m=0}^{n}(-1)^{n-m+l}S_1(n,m+l)\binom {m+l}{l} \frac{(-x)^l}{[m+1]_{p,q}^k}\\
&=\sum_{m=0}^{n}\frac{(-1)^{n-m}}{[m+1]_{p,q}^k} \sum_{l=0}^m \binom {m+l}{l} S_1(n,m+l) x^l\\
&=\sum_{m=0}^{n}\frac{(-1)^{n-m}}{[m+1]_{p,q}^k} S_1(n,m,x). \qedhere
\end{align*}
\end{proof}

It is not difficult to give a $(p,q)$-analogue of \eqref{genqcauhcy}.

\begin{theorem}\label{gfuncau}
The exponential generating function of the $(p,q)$-poly--Cauchy polynomials  $C_{n,p,q}^{(k)}(x)$ is
\begin{align}
\frac{\lif_{k,p,q}\left(\ln(1+t)\right)}{(1+t)^x}=\sum_{n=0}^\infty C_{n,p,q}^{(k)}(x)\frac{t^n}{n!},\label{ec3cauchy}
\end{align}
where
\begin{align}
\lif_{k,p,q}(t)=\sum_{n=0}^\infty\frac{t^n}{n![n+1]_{p,q}^k}
\end{align}
is the $k$-th $(p, q)$-polylogarithm factorial function.
\end{theorem}

\begin{proof}
From Theorem \ref{teocaupq} we have
\begin{align*}
\sum_{n=0}^\infty C_{n,p,q}^{(k)}(x)\frac{t^n}{n!}&=\sum_{n=0}^\infty \sum_{m=0}^n(-1)^{n-m}S_1(n,m)\sum_{l=0}^m\binom ml \frac{(-x)^l}{[m-l+1]_{p,q}^k} \frac{t^n}{n!}\\
&=\sum_{m=0}^\infty  \sum_{n=m}^\infty (-1)^{n-m}S_1(n,m)\frac{t^n}{n!} \sum_{l=0}^m\binom ml \frac{(-x)^l}{[m-l+1]_{p,q}^k} \\
&=\sum_{m=0}^\infty  \frac{(\ln(1+t))^m}{m!} \sum_{l=0}^m \binom ml \frac{(-x)^l}{[m-l+1]_{p,q}^k} \\
&=\sum_{l=0}^\infty  \frac{(-x)^l}{l!} \sum_{m=l}^\infty \frac{(\ln(1+t))^m}{(m-l)! [m-l+1]_{p,q}^k} \\
&=\sum_{l=0}^\infty  \frac{(-x)^l}{l!} \sum_{n=0}^\infty \frac{(\ln(1+t))^{n+l}}{n! [n+1]_{p,q}^k} \\
&=\frac{1}{(1+t)^x}\sum_{n=0}^\infty  \frac{(\ln(1+t))^n}{n! [n+1]_{p,q}^k} \\
&=\frac{\lif_{k,p,q}\left(\ln(1+t)\right)}{(1+t)^x}. \qedhere
\end{align*}
\end{proof}

Similarly,  we can defined the $(p,q)$-poly-Cauchy polynomials of the second kind by
\begin{align*}
\widehat{C}_{n,p,q}^{(k)}(x)&=\underbrace{\int_{0}^1 \cdots \int_{0}^1}_{k}(-t_1\cdots t_k+x)_nd_{p,q}t_1\cdots d_{p,q}t_k.
\end{align*}
We can find analogous  expressions to \eqref{ec1cauchy}, \eqref{ec2cauchy} and \eqref{ec3cauchy}.

\begin{theorem}\label{sticau2}
If $n\geq 1$, we have
\begin{align}
\widehat{C}_{n,p,q}^{(k)}(x)&=(-1)^{n}\sum_{m=0}^nS_1(n,m)\sum_{l=0}^m\binom ml \frac{(-x)^l}{[m-l+1]_{p,q}^k} \\
&=(-1)^n\sum_{m=0}^nS_1\left(n,m,-x\right)\frac{1}{[m+1]_{p,q}^k}.
\end{align}
Moreover, the exponential generating function of the $(p,q)$-poly--Cauchy polynomials  $\widehat{C}_{n,p,q}^{(k)}(x)$ is
\begin{align*}
(1+t)^x\lif_{k,p,q}\left(-\ln(1+t)\right)=\sum_{n=0}^\infty \widehat{C}_{n,p,q}^{(k)}(x)\frac{t^n}{n!}.
\end{align*}
\end{theorem}

\subsection{Some relations between $(p,q)$-poly-Bernoulli polynomials and $(p,q)$-poly-Cauchy polynomials}

The weighted Stirling numbers satisfy the following orthogonality relation \cite{Car}:
$$\sum_{l=m}^n(-1)^{n-l}S_2(n,l,x)S_1(l,m,x)=\sum_{l=m}^n(-1)^{l-m}S_1(n,l,x)S_2(l,m,x)=\delta_{m,n},$$   
where $\delta_{m,n}=1$ if $m=n$ and 0 otherwise.  From above relations we obtain the inverse relation:
$$f_n=\sum_{m=0}^n(-1)^{n-m}S_1(n,m,x)g_m  \iff g_n=\sum_{m=0}^nS_2(n,m,x)f_m.$$

\begin{theorem}
The $(p,q)$-poly-Bernoulli polynomials and $(p,q)$-poly-Cauchy polynomials of both kinds satisfy the following relations
\begin{align}
\sum_{m=0}^nS_1(n,m,x)B_{m,p,q}^{(k)}(x)&=\frac{n!}{[n+1]_{p,q}^k}, \label{invrel1}\\
\sum_{m=0}^nS_2(n,m,x)C_{m,p,q}^{(k)}(x)&=\frac{1}{[n+1]_{p,q}^k}\label{invrel2} ,\\
\sum_{m=0}^nS_2(n,m,-x)\widehat{C}_{m,p,q}^{(k)}(x)&=\frac{(-1)^n}{[n+1]_{p,q}^k} \label{invrel3}.
\end{align}
\end{theorem}
\begin{proof}
From Theorem \ref{polyberrel1} and the inverse relation for the weighted Stirling numbers with
$$f_m=\frac{(-1)^mm!}{[m+1]_{p,q}^k}, \quad \text{and} \quad g_n=(-1)^nB_{n,p,q}^{(k)}(x),$$
we obtain the identity \eqref{invrel1}.  The remaining relations can be verified in a similar way by using  Theorems \ref{teocaupq} and \ref{sticau2}. 
\end{proof}
Note that if $p\to 1$ we obtain Theorem 6  in \cite{Kom4}.
\begin{theorem}
The $(p,q)$-poly-Bernoulli polynomials and $(p,q)$-poly-Cauchy polynomials of both kinds satisfy the following relations
\begin{align}
B_{n,p,q}^{(k)}(x)&=\sum_{l=0}^n\sum_{m=0}^n(-1)^{n-m}m!S_2(n,m,x)S_2(m,l,y)C_{l,p,q}^{(k)}(y),\\
B_{n,p,q}^{(k)}(x)&=\sum_{l=0}^n\sum_{m=0}^n(-1)^{n}m!S_2(n,m,x)S_2(m,l,-y)\widehat{C}_{l,p,q}^{(k)}(y),\\
C_{n,p,q}^{(k)}(x)&=\sum_{l=0}^n\sum_{m=0}^n\frac{(-1)^{n-m}}{m!}S_1(n,m,x)S_1(m,l,y)B_{l,p,q}^{(k)}(y), \label{repolycabe}\\
\widehat{C}_{n,p,q}^{(k)}(x)&=\sum_{l=0}^n\sum_{m=0}^n\frac{(-1)^{n}}{m!}S_1(n,m,-x)S_1(m,l,y)B_{l,p,q}^{(k)}(y).
\end{align}
\end{theorem}
\begin{proof}
We only show the proof of \eqref{repolycabe}. The proofs of the remaining identities are similar.  
From Equations  \eqref{ec2cauchy} and \eqref{invrel1} we have
 \begin{align*}
 \sum_{l=0}^n\sum_{m=0}^n\frac{(-1)^{n-m}}{m!}&S_1(n,m,x)S_1(m,l,y)B_{l,p,q}^{(k)}(y)\\
 &=\sum_{m=0}^n \frac{(-1)^{n-m}}{m!}S_1(n,m,x) \sum_{l=0}^mS1(m,l,y)B_{l,p,q}^{(k)}(y)\\
 &=\sum_{m=0}^n \frac{(-1)^{n-m}}{m!}S_1(n,m,x)\frac{m!}{[m+1]_{p,q}^k}\\
&=C_{n,p,q}^{(k)}(x).\qedhere\
  \end{align*}
\end{proof}

Finally, we show some relations between $(p,q)$-poly-Cauchy  polynomials of both kinds. 
\begin{theorem}
If $n\geq 1$ we have
\begin{align}
(-1)^n\frac{C_{n,p,q}^{(k)}(x)}{n!}&=\sum_{m=1}^n\binom{n-1}{m-1}\frac{\widehat{C}_{m,p,q}^{(k)}(x)}{m!},\\
(-1)^n\frac{\widehat{C}_{n,p,q}^{(k)}(x)}{n!}&=\sum_{m=1}^n\binom{n-1}{m-1}\frac{C_{m,p,q}^{(k)}(x)}{m!}. \label{rel1caufin}
\end{align}
\end{theorem}
\begin{proof}

From definition of the $(p,q)$-poly-Cauchy polynomials of the first kind we get
\begin{align*}
(-1)^n\frac{C_{n,p,q}^{(k)}(x)}{n!}&=(-1)^n\underbrace{\int_{0}^1 \cdots \int_{0}^1}_{k}\frac{(t_1\cdots t_k-x)_n}{n!}d_{p,q}t_1\cdots d_{p,q}t_k\\
&=(-1)^n\underbrace{\int_{0}^1 \cdots \int_{0}^1}_{k}\binom{t_1\cdots t_k-x}{n}d_{p,q}t_1\cdots d_{p,q}t_k\\
&=\underbrace{\int_{0}^1 \cdots \int_{0}^1}_{k}\binom{x-t_1\cdots t_k+n-1}{n}d_{p,q}t_1\cdots d_{p,q}t_k
\end{align*}
By using the Vandermonde convolution
$$\sum_{k=0}^n\binom{r}{k}\binom{s}{n-k}=\binom{r+s}{n},$$  
with $r=x-t_1\cdots t_k$ and $s=n-1$ we obtain
\begin{align*}
(-1)^n\frac{C_{n,p,q}^{(k)}(x)}{n!}&=\underbrace{\int_{0}^1 \cdots \int_{0}^1}_{k} \sum_{l=0}^n
\binom{x-t_1\cdots t_k}{l}\binom{n-1}{n-l}d_{p,q}t_1\cdots d_{p,q}t_k\\
&=\sum_{l=0}^n \binom{n-1}{n-l} \underbrace{\int_{0}^1 \cdots \int_{0}^1}_{k} 
\binom{x-t_1\cdots t_k}{l}d_{p,q}t_1\cdots d_{p,q}t_k\\
&=\sum_{l=0}^n \binom{n-1}{n-l} \frac{1}{l!}\underbrace{\int_{0}^1 \cdots \int_{0}^1}_{k} 
(-t_1\cdots t_k+x)_ld_{p,q}t_1\cdots d_{p,q}t_k\\
&=\sum_{l=0}^n \binom{n-1}{n-l} \frac{\widehat{C}_{l,p,q}^{(k)}(x)}{l!}.
\end{align*}
The proof of \eqref{rel1caufin}  is similar. 
\end{proof}


\end{document}